\documentclass[a4paper,12pt]{article}

\usepackage{amsmath, wrapfig}
\usepackage{dsfont, a4wide, amsthm, amssymb, amsfonts, graphicx}
\usepackage{fancyhdr, xspace, psfrag, setspace, supertabular, color, enumerate}
\usepackage{hyperref}
\usepackage{bbm}

\newtheorem{theorem}{Theorem}[section]
\newtheorem{proposition}[theorem]{Proposition}

\newtheorem{conjecture}[theorem]{Conjecture}

\newtheorem*{example*}{Example}

\def\d{\delta}
\def\e{\epsilon}

\def\ca{\mathcal{A}}

\def\cc{\mathcal{C}}
\def\cd{\mathcal{D}}
\def\cf{\mathcal{F}}
\def\ct{\mathcal{T}}
\def\cp{\mathcal{P}}
\def\cs{\mathcal{S}}
\def\cg{\mathcal{G}}

\def\cw{\mathcal{W}}
\def\ci{\mathcal{I}}
\def\Z{\mathbb{Z}}
\def\F{\mathbb{F}}

\def\E{\mathbb{E}}
\def\P{\mathbb{P}}

\DeclareMathOperator{\Var}{Var}

\title{Three applications of coverings to difference patterns}

\author{Thomas Karam\footnote{Mathematical Institute, University of Oxford. Email: \texttt{thomas.karam@maths.ox.ac.uk}.}}

\begin{document}
\maketitle

\begin{abstract}

We show that a conceptually simple covering technique has surprisingly rich applications to density theorems and conjectures on patterns in sets involving set differences. These applications fall into three categories: (i) analogues of these statements to distance $2$ versions of the pattern, (ii) reduction of these statements to relative versions, and (iii) reductions of these statements to a quasirandom case with respect to some quantities that affect the number of realisations of the pattern.

\end{abstract}

\tableofcontents

\section{Introduction}

Throughout this paper we will use the following notations. If $n$ is a positive integer then we will write $[n]$ for the set $\{1,\dots,n\}$ of positive integers between $1$ and $n$. We will often use notations such as $[n]^{d_1} \cup \dots \cup [n]^{d_s}$, where $s,d_1, \dots, d_s$ are positive integers; unless stated otherwise they will always be understood as disjoint unions: for instance, $[n] \cup [n]$ will refer not to $[n]$ but to the disjoint union of two copies of $[n]$. We will also often write a subset $A \subset [n]^{d_1} \cup \dots \cup [n]^{d_s}$ as $A_1 \cup \dots \cup A_s$. In such decompositions it will be implicit that $A_i$ is a subset of the $i$th part $[n]^{d_i}$ of the union for every $i$, unless stated otherwise. If $A,B$ are two finite sets with $A \subset B$ and $B \neq \emptyset$, then we will refer to the ratio $|A|/|B|$ as the \emph{density of $A$ inside $B$}, or as the \emph{density of $A$} when there is no ambiguity as to which set $B$ is being considered.

\subsection{Background on patterns in set systems}

Ramsey theory largely consists in statements broadly asserting that if a subcollection of a collection of objects is large enough, then we can find some configuration of elements (usually some pair of elements, or some $k$-tuple of elements for some fixed $k$) of the subcollection that satisfies some desired property. A famous example is Ramsey’s theorem (\cite{Ramsey}, discovered in 1930) which says that if $k$ is an integer and a large enough complete graph $K_n$ has its edges each coloured in red or blue then it must eventually contain a monochromatic complete graph $K_k$ (see \cite{Campos Griffiths Morris Sahasrabudhe} for a recent breakthrough of Campos, Griffiths, Morris and Sahasrabudhe on the upper bound on the smallest such $n$, and \cite{Gupta Ndiaye Norin Wei} for a further quantitative improvement by Gupta, Ndiaye, Norin and Wei). Within these statements, a large class of results is known as that of density theorems. There, the assumptions are that the subcollection contains at least a fixed positive proportion of the objects of the ambient collection, and then that the size of the whole collection is large enough depending on this proportion. One of the celebrated such statements is Szemer\'edi’s theorem \cite{Szemeredi} first proved in 1975.

\begin{theorem}\label{Szemeredi’s theorem}

Let $k$ be a positive integer and let $\d>0$. If $n$ is large enough depending on $k,\d$ only, then every subset $A\subset[n]$ with size at least $\d n$ contains an arithmetic progression of length $k$.

\end{theorem}

The present paper will focus on density theorems as opposed to colouring theorems. For instance, Szemer\'edi’s theorem is viewed as the density version of van der Waerden’s theorem \cite{van der Waerden}, which states that if $r,k$ are positive integers then any colouring of the integers with $r$ colours contains an arithmetic progression with length $k$. Conversely, van der Waerden’s theorem is viewed as the colouring version of Szemer\'edi’s theorem. For any positive integer $n$, colouring the integers of $[n]$ with $r$ colours ensures in particular that one colour is used for at least a proportion $1/r$ of these integers, and Szemer\'edi’s theorem hence implies van der Waerden’s theorem. More generally, the density version of a statement usually implies its colouring version for the same reason. In the converse direction however, the density version of a colouring result is not always true, as can be seen from several examples starting with Ramsey’s theorem itself. Indeed if $n$ is an integer which for simplicity we require to be even, then the complete bipartite graph $K_{n/2, n/2}$ contains $n^2/4$ edges, so in particular at least a quarter of the total number of edges of $K_n$, but does not contain any triangle, and that remains true however large $n$ is taken to be. As for the results for which the density version is true, the proof of the density version is usually much more involved than that of the colouring version. Again, this may already be seen to be the case with the proofs of Szemer\'edi’s theorem, compared to those of van der Waerden’s theorem.

Besides Szemer\'edi’s theorem, there are several density results which have received attention and for which the matter of the optimal bounds is an active research topic. For instance, obtaining reasonable bounds for a multidimensional version of Szemer\'edi’s theorem is still a wide open problem, and recent work of Peluse \cite{Peluse} establishes such bounds in a first special case of “$L$-shaped" configurations.

But the purely qualitative side of density theorems is also still far from completely understood. For instance, Gowers \cite{Gowers} describes the following conjecture (in its version with $d_1=1, \dots, d_s=s$) as a central open problem in Ramsey theory.

\begin{conjecture} \label{polynomial density Hales-Jewett conjecture}
Let $k, s, d_1, \dots, d_s$ be positive integers and let $\d>0$. If $n$ is large enough depending on $k, s, d_1, \dots, d_s, \d$ only then for every subset $A$ of the set \[[k]^{[n]^{d_1}}\times\dots\times[k]^{[n]^{d_s}}\] with density at least $\d$ there is a non-empty subset $S\subset [n]$ and an element \[y \in ([n]^{d_1} \setminus S^{d_1}) \cup \dots \cup ([n]^{d_s} \setminus S^{d_s})\] such that whenever the coordinates of $x$ are the same within each of the sets $S^{d_1}, \dots, S^{d_s}$, and coincide with those of $y$ outside these sets, we have that $x \in A$.
\end{conjecture}

One motivation to consider Conjecture \ref{polynomial density Hales-Jewett conjecture} is that it would simultaneously imply both the Bergelson-Leibman theorem and the density Hales-Jewett theorem, two generalisations of Szemer\'edi’s theorem. An arithmetic progression $\{x, x+d, \dots, x+(k-1)d\}$ can be viewed as merely a special case of a set of the type $\{x, x+P_1(d), \dots, P_{k-1}(d)\}$ where $P_1, \dots, P_{k-1}$ are specified polynomials, and Bergelson and Leibman \cite{Bergelson Leibman} correspondingly extended Szemer\'edi’s theorem.

\begin{theorem}\label{Bergelson-Leibman theorem} Let $k$ be a positive integer, let $\d > 0$ and let $P_1,\dots,P_{k-1}$ be polynomials with integer coefficients and no constant term. If $n$ is large enough depending on $k, \d, P_1, \dots, P_{k-1}$ only then for every subset $A$ of $[n]$ with density at least $\d$ there exist positive integers $a$ and $d\ne 0$ such that $A$ contains all elements of \[\{a, a+P_1(d),\dots,a+P_{k-1}(d)\}.\] \end{theorem}

Meanwhile, the density Hales-Jewett theorem, proved by Furstenberg and Katznelson (\cite{Furstenberg and Katznelson k=3}, \cite{Furstenberg and Katznelson}) using ergodic theory and then more combinatorially by the Polymath1 project \cite{Polymath}, generalises Szemer\'edi’s theorem by replacing subsets of integers by subsets of high-dimensional sets.

\begin{theorem} \label{density Hales-Jewett theorem}
Let $k$ be a positive integer and let $\d>0$. If $n$ is large enough depending on $k,\d$ only, then every subset $A$ of $[k]^n$ with density at least $\d$ contains some $k$-tuple of points $\{x^1, \dots, x^k\}$ such that for some non-empty subset $W$ of $[n]$ we have for all $j \in [d]$ and $i \in [n]$ that $x_i^j = j$ if $i \in W$ and $x_i^j = x_i^1$ otherwise.
\end{theorem}



Since Conjecture \ref{polynomial density Hales-Jewett conjecture} implies Theorem \ref{Szemeredi’s theorem}, Theorem \ref{Bergelson-Leibman theorem} and Theorem \ref{density Hales-Jewett theorem}, and the proofs of the latter three theorems are difficult, it appears natural to attempt to look for new basic difficulties that do not arise in any of these three theorems and treat them in isolation before attempting a solution to Conjecture \ref{polynomial density Hales-Jewett conjecture}.

Taking $k=2$ in Conjecture \ref{polynomial density Hales-Jewett conjecture} is such a step. First, it does away with the structure of arithmetic progression and its accompanying difficulties, and the special cases $k=2$ of Theorem \ref{Szemeredi’s theorem}, Theorem \ref{Bergelson-Leibman theorem} and Theorem \ref{density Hales-Jewett theorem} are degenerate or at least much easier to prove. Indeed the first two follow from a double-counting argument, and the statement of the third then becomes the same as that of Sperner’s theorem \cite{Sperner}, which states that a dense subcollection of subsets of $[n]$ necessarily contains two distinct sets such that one is contained in the other (provided that $n$ is large enough). Second, the special case $k=2$ of Conjecture \ref{polynomial density Hales-Jewett conjecture}, which we are about to state, appears to on its own be a difficult problem.

\begin{conjecture}\label{case k=2 of polynomial DHJ}

Let $s, d_1, \dots, d_s$ be positive integers and let $\d>0$. Then, for $n$ large enough depending on $s, d_1, \dots, d_s, \d$ only, every subset of $\cp([n]^{d_1} \cup \dots \cup [n]^{d_s})$ that has density at least $\d$ contains all $2^d$ sets \[A \cup \bigcup_{r \in T} S^r\] where $T \subset [s]$, the unions are disjoint unions, and the sets $\emptyset \neq S \subset [n]$ and $A \subset[n]^{d_1} \cup \dots \cup [n]^{d_s}$ are common to all $2^d$ sets.

\end{conjecture}

On the way to Conjecture \ref{case k=2 of polynomial DHJ}, we may first ask for two of the required $2^d$ sets, and to obtain two such sets it suffices to obtain the point corresponding to $T$ empty and another point. Whatever second point we require, Conjecture \ref{case k=2 of polynomial DHJ} in turn specialises as follows (with parameters $s, d_1, \dots, d_s$ that might be lower than originally), to a polynomial set difference.

\begin{conjecture}\label{polynomial difference conjecture}

Let $s, d_1, \dots, d_s$ be positive integers and let $\d>0$. Then, for $n$ large enough depending on $s, d_1, \dots, d_s,\d$ only, every subset of $\cp([n]^{d_1} \cup \dots \cup [n]^{d_s})$ that has density at least $\d$ contains a pair $(A,B)$ of distinct subsets of $[n]^{d_1} \cup \dots \cup [n]^{d_s}$ such that $A \subset B$ and \[B \setminus A = S^{d_1} \cup \dots \cup S^{d_s}\] for some $S \subset [n]$.

\end{conjecture}

Taking $s=1$, $d_1=2$ we recover the following special case of Conjecture \ref{polynomial density Hales-Jewett conjecture} and Conjecture \ref{polynomial difference conjecture}, which is a conjecture of Gowers \cite{Gowers} that requires a square difference.

\begin{conjecture}\label{square difference conjecture}

Let $\d>0$. Then, for $n$ large enough depending on $\d$, every subset $\ca$ of $\cp([n]^2)$ that has density at least $\d$ contains a pair $(A,B)$ of distinct subsets of $[n]^2$ such that $A$ is contained in $B$ and $B \setminus A = S^2$ for some $S \subset [n]$.

\end{conjecture}

Let us also mention a related conjecture from \cite{Gowers}: a graph theoretic version of Conjecture \ref{square difference conjecture} requiring clique differences.

\begin{conjecture}\label{clique difference conjecture}

Let $\d>0$. Then, for $n$ large enough depending on $\d$, every subset of the collection of \emph{non-oriented} graphs on the vertex set $[n]$ that has density at least $\d$ contains a pair $(H,G)$ of distinct graphs such that $H$ is a subgraph of $G$ and the complement of $H$ inside $G$ is a clique.

\end{conjecture}

Throughout, we will refer to pairs $(A,B)$, (resp. $(A,B)$, resp. $(H,G)$), satisfying the conclusions of Conjecture \ref{polynomial difference conjecture} (resp. Conjecture \ref{square difference conjecture}, resp. Conjecture \ref{clique difference conjecture}) as \emph{polynomial difference pairs}, \emph{square difference pairs}, and \emph{clique difference pairs}. Furthermore, for any positive integer $d$, we will refer to a pair $(A,B)$ of distinct subsets of $[n]^d$ satisfying $A \subset B$ and $B \setminus A = S^d$ for some $S \subset [n]$ as a \emph{$d$th power difference pair}, or rather as a \emph{power difference pair} when there is no ambiguity as to the value of $d$.

Conjecture \ref{square difference conjecture} implies Conjecture \ref{clique difference conjecture}, since if $\cg$ is a set of subgraphs on the set on $[n]$ vertices, then the collection $\ca$ of subsets $A \subset [n]^2$ such that $A(x,y) = \mathbbm{1}_{\{x,y\} \in E(G)}$ whenever $1 \le x < y \le n$ for some $G \in \cg$ has the same density inside $\cp([n]^2)$ as $\cg$ has inside the collection of subgraphs with vertex set $[n]$. If that density is some fixed $\d$, then for $n$ large enough Conjecture \ref{square difference conjecture} provides a square difference pair $(A,B) \in \ca$, which, by restricting to $\{(x,y) \in [n]^2: 1 \le x < y \le n \}$ ensures that $\cg$ contains a clique difference pair.

A possible approach to Conjecture \ref{clique difference conjecture}, based on imitating the proof of Sperner’s theorem had been suggested by Gowers \cite{Gowers} and was later ruled out by Alweiss \cite{Alweiss}. More recently still, Conjecture \ref{clique difference conjecture} has been approached by Alon \cite{Alon}, within the framework of graph-codes, that is, of sets of graphs such that the symmetric difference of any two of them avoids a specified set of graphs.

\subsection{Main results}

This paper will primarily be focused around a covering argument which allows us to obtain surprisingly diverse conclusions related to the last few conjectures, and more broadly to set difference patterns.

With Conjecture \ref{polynomial difference conjecture} in sight, we may define an \emph{oriented} graph with vertex set $\cp([n]^{d_1} \cup [n]^{d_2} \cup \dots \cup [n]^{d_s})$, and join $A$ to $B$ by an edge if $(A,B)$ constitutes a square difference pair. The statement of Conjecture \ref{polynomial difference conjecture} can then be reformulated as stating that every dense subset $\ca$ of the vertex set contains (for $n$ large enough) a pair of vertices that is joined by an edge. We shall not prove this but we shall instead prove that $\ca$ must contain a pair $(A,B)$ of vertices such that for some $U$ (not necessarily in $\ca$) both $(U,A)$, $(U,B)$ are edges of the oriented graph. In particular, in the nonoriented version of the presently defined graph, the vertices $A$ and $B$ are at distance at most $2$ from one another.

In order to state the first main result that we shall prove, we first introduce some auxiliary notation. If $m$ is a positive integer, and $I = [l+1, \dots, l+m] \subset [n]$ is an interval of size $m$ for some $l \ge 0$, we define for every subset $A \subset I^{d_1} \cup \dots \cup I^{d_s}$ a subset $h_I(A)$ of $[m]^{d_1} \cup \dots \cup [m]^{d_s}$ by replacing the coordinates $l+1, \dots, l+m$ by $1, \dots, m$ respectively in the elements of $A$. Then, for every $\cf \subset \cp([m]^{d_1} \cup \dots \cup [m]^{d_s})$ we define $\cf_{I} \subset \cp(X^{d_1} \cup \dots \cup X^{d_s})$ by \[\cf_{I} = \{A \subset X^{d_1} \cup \dots \cup X^{d_s}: h_{I}(A) \in \cf\}.\]

\begin{theorem}\label{distance 2 theorem}

Let $s, d_1, \dots, d_s, m$ be positive integers, let $\d>0$ and let $\cf$ be a collection of subsets of $[m]^{d_1} \cup \dots \cup [m]^{d_s}$ with $|\cf| \ge 4 \d^{-1}$. If $n$ is large enough depending on $s,d_1,\dots,d_s,m,\d$ only, then every collection $\ca \subset \cp([n]^{d_1} \cup \dots \cup [n]^{d_s})$ with density at least $\d$ contains a pair $(A,B)$ of distinct subsets of $[n]^{d_1} \cup \dots \cup [n]^{d_s}$ such that for some $U \in \cp([n]^{d_1} \cup \dots \cup [n]^{d_s})$ (not necessarily in $\ca$) we have $U \subset A$, $U \subset B$, $A \setminus U = F_1$, $B \setminus U = F_2$ where $F_1,F_2 \in \cf_{I}$ for some interval $I \subset [n]$ of size $m$. We may instead require $F_1 \in \cf_{I_1}$ and $F_2 \in \cf_{I_2}$ for some disjoint intervals $I_1, I_2 \subset [n]$ of size $m$, leading to $A \Delta B = F_1 \cup F_2$. Alternatively, if $\cf$ is nested, then we may require $F_2 \subset F_1$, leading to $A \Delta B = F_1 \setminus F_2.$ \end{theorem}

\begin{theorem}\label{corollary of distance 2 theorem}

Let $s, d_1, \dots, d_s, m$ be positive integers and let $\d>0$. Then, for $n$ large enough depending on $s, d_1, \dots, d_s, m, \d$ only, every subset $\ca$ of $\cp([n]^{d_1} \cup \dots \cup [n]^{d_s})$ that has density at least $\d$ contains a pair $(A, B)$ of distinct subsets such that for some $U \in \cp([n^{d_1} \cup \dots \cup [n]^{d_s})$ (not necessarily in $\ca$) we have $U \subset A$, $U \subset B$ and \begin{align*} A \setminus U &= S_1^{d_1} \cup \dots \cup S_1^{d_s} \\ B \setminus U &= S_2^{d_1} \cup \dots \cup S_2^{d_s}\end{align*} for some $S_1, S_2 \subset [n]$. We may furthermore require $S_1$, $S_2$ to both have size $m$ and $S_1$ to be disjoint from $S_2$, which provides \[A \Delta B = (S_1^{d_1} \cup S_2^{d_1}) \cup \dots \cup (S_1^{d_s} \cup S_2^{d_s}). \] Alternatively we may require $S_2$ to be a strict subset of $S_1$, which provides \[A \Delta B = (S_1^{d_1} \setminus S_2^{d_1}) \cup \dots \cup (S_1^{d_s} \setminus S_2^{d_s}).\]

\end{theorem}

We note that this suffices to establish Conjecture \ref{polynomial density Hales-Jewett conjecture} in the case $d_1=\dots=d_s=1$ (that is, a “simultaneous Sperner theorem”), but not otherwise, because squares and higher powers of sets are not stable under any of the union, set difference, and symmetric difference operations.

The second main result that we will obtain is an equivalence between several versions of Conjecture \ref{polynomial difference conjecture}. To state it we introduce a few more notations. If $d$ is a positive integer, then we say that a subset $A \subset [n]^d$ is \emph{symmetric} if \[\mathbbm{1}_A(x_1, \dots, x_d) = \mathbbm{1}_A(x_{\sigma(1)}, \dots, x_{\sigma(d)})\] for every permutation $\sigma$ of $[d]$, and write $\cp([n]^d)_{\mathrm{Sym}}$ for the collection of symmetric subsets of $\cp([n]^{d})$. For every subset $S \subset [n]$ and every positive integer $e$ we write $K(S,e)$ for the complete $e$-uniform hypergraph on the vertex set $S$. If $s, d_1, \dots, d_s$ are positive integers then we write $\cg(K([n],d_1) \cup \dots \cup K([n],d_s))$ for the collection of disjoint unions $G_1 \cup \dots \cup G_s$ of subhypergraphs $G_i \subset K([n],d_i)$.

\begin{theorem}\label{statements equivalent to the polynomial difference conjecture}

Conjecture \ref{polynomial difference conjecture} is equivalent to each of the following statements. \begin{enumerate}[(i)]

\item Let $s,d$ be positive integers and let $\d>0$. Then, for $n$ large enough depending on $d$ and $\d$ only, every subset $\ca$ of $\cp([n]^d \cup \dots \cup [n]^d)$ that has density at least $\d$ contains a pair $(A,B)$ of distinct subsets of $[n]^d \cup \dots \cup [n]^d$ such that $A$ is contained in $B$ and $B \setminus A = S^d \cup \dots \cup S^d$ for some $S \subset [n]$, with all disjoint unions taken over $s$ copies.

\item Let $d$ be a positive integer and let $\d>0$. Then, for $n$ large enough depending on $d$ and $\d$ only, every subset $\ca$ of $\cp([n]^d)$ that has density at least $\d$ contains a pair $(A,B)$ of distinct subsets of $[n]^d$ such that $A$ is contained in $B$ and $B \setminus A = S^d$ for some $S \subset [n]$.

\item Let $d$ be a positive integer. There exists a sequence $(\ca_m)_{m \ge 1}$ of subsets $\ca_m \subset [m]^{d}$ such that for every $\d>0$, if $m$ is large enough (depending on $d$, $\d$ and on the sequence $(\ca_m)_{m \ge 1}$) then every $\ca \subset \ca_m$ with $|\ca| \ge \d |\ca_m|$ contains a pair $(A,B)$ of distinct subsets of $[m]^d$ satisfying $A \subset B$ and $B \setminus A = S^d$ for some $S \subset [n]$.

\item Let $d$ be a positive integer and let $\d>0$. Then, for $n$ large enough depending on $d$ and $\d$ only, every subset $\ca$ of $\cp([n]^d)_{\mathrm{Sym}}$ that has density at least $\d$ contains a pair $(A,B)$ of distinct subsets of $[n]^d$ such that $A$ is contained in $B$ and $B \setminus A = S^d$ for some $S \subset [n]$.

\item Let $s, d_1, \dots, d_s$ be positive integers and let $\d>0$. Then, for $n$ large enough depending on $s, d_1, \dots, d_s,\d$ only, every subset of $\cg(K([n],d_1) \cup \dots \cup K([n],d_s))$ that has density at least $\d$ contains a pair $(H,G)$ of distinct subsets of $\cg(K([n],d_1) \cup \dots \cup K([n],d_s))$ such that $H \subset G$ and \[G \setminus H = K(S,d_1) \cup \dots \cup K(S,d_s)\] for some $S \subset [n]$.

\end{enumerate}

\end{theorem}

Among the statements in the equivalence above, we emphasise (iii), which states that it suffices for \emph{one} relative version of the power difference statement (ii) to hold for Conjecture \ref{polynomial difference conjecture} to hold. It may seem surprising at first that we seek to reduce the “absolute” version (ii) \emph{to} this relative version, as it is has often been the case that a relative version of a density theorem is reduced \emph{to} the absolute version (through some kind of transference principle), with this reduction going in the direction opposite to the one that we are drawing attention to. (For instance, the reduction of dense subsets of the primes to dense subsets of the integers is one of the fundamental ideas behind the proof by Green and Tao \cite{Green and Tao} that dense subsets of the primes contain arbitrarily long arithmetic progressions.) Taking note of (iv) as a special case of (iii) however sheds light on how (iii) may be meaningfully used, by allowing us to place essentially arbitrary additional requirements on the dense subsets of $\cp([n]^d)$ that we work with (in the case of (iv), that the sets are symmetric). Unsurprisingly, not all such additional requirements are helpful, and correspondingly in the equivalence between (ii) and (iii) we cannot strengthen the existential quantifier in front of $(\ca_m)$ to a universal quantifier: for instance, if in (iii) we take $\ca_m$ to be the family \[\{S^d: \emptyset \neq S \subset [m]\}\] then for every $d \ge 2$ there are no power difference pairs in $\ca_m$. Rather, the challenge is to find subsets $\ca_m \subset \cp([m]^d)$ with size tending to infinity with $m$ and such that for $m$ large it becomes substantially easier to prove the existence of power differences in dense subsets of $\ca_m$ than in dense subsets of $[m]^d$.

To prove that the graph-theoretic statement (v) is equivalent to Conjecture \ref{polynomial difference conjecture} it will be convenient for us to go through (iv) as this will lead to a proof that is somewhat simpler to write, but the proof of this equivalence can also be done by directly generalising the deduction that we give and without using any of (i)-(iv). Nonetheless, we note that the equivalence between (ii) and (iv) shows in particular that the difference between the square difference and clique difference conjectures, Conjecture \ref{square difference conjecture} and Conjecture \ref{clique difference conjecture}, is smaller than it might appear at first. Conjecture \ref{square difference conjecture} does not reduce to Conjecture \ref{clique difference conjecture} because of the diagonal, but it does reduce to a slight modification of the latter where graphs are no longer loopless and a clique is defined to include loops at all of its vertices. Moreover, the equivalent statements (i) and (v) are respectively extensions of Conjecture \ref{square difference conjecture} and Conjecture \ref{clique difference conjecture}.

Finally, our third main result will involve reducing Conjecture \ref{polynomial difference conjecture} to the case where linear forms $\F_p^n \to \F_p$ with $p$ some small prime integer do not by themselves suggest an unexpected number of polynomial difference pairs. Having previously reduced Conjecture \ref{polynomial difference conjecture} to the case $s=1$ (in (i), Theorem \ref{statements equivalent to the polynomial difference conjecture}), we will restrict ourselves to the case of power difference pairs (rather than polynomial difference pairs).

Linear forms $\F_p^n \to \F_p$ have a strong effect on how likely a pair $(A,B)$ is to constitute a power difference pair. If $d \ge 2$ is an integer, $p$ is a prime, and $\phi: \F_p^n \to \F_p$ is a linear form, defined by \[\phi(x) = a_1 x_1 + \dots + a_n x_n\] for every $x \in \F_p^n$ then we can assign to $\phi$ another linear form $\Phi: \F_p^{[n]^{d}} \to \F_p$ defined by \begin{equation} \Phi(x) = \sum_{i_1, \dots, i_d \in [n]} a_{i_1} \dots a_{i_d} x_{(i_1, \dots, i_d)} \label{definition of Phi in terms of phi} \end{equation} for every $x \in \F_p^{[n]^{d}}$. If $A \subset [n]^d$ then we can define $\Phi(A)$ by identifying $A$ with the indicator function $1_A$ modulo $p$, that is, we define \[\Phi(A) = \sum_{i_1,\dots,i_d \in [n]} a_{i_1} \dots a_{i_d} 1_A(i_1, \dots, i_d)\] and likewise define $\phi(S)$ when $S \subset [n]$ by identifying $S$ with the indicator function $1_S$ modulo $p$. If $A,B$ are subsets of $[n]^d$ such that $A \subset B$ and $B \setminus A = S^d$ for some non-empty $S \subset [n]$, then in particular \[\Phi(B) - \Phi(A) = \Phi(B \setminus A) = \Phi(S^d) = \phi(S)^d\] is a quadratic residue modulo $p$. 

One can go further still: if $u$,$v$ are elements of $\F_p$, $(A,B)$ is a power difference pair (with $A \neq B$ as usual) selected uniformly at random, and $\cs$ is the resulting distribution of the set $S$ in the partition $B = A \cup S^d$, then the probability $P_{u,v}$ that $\Phi(A) = u$ and $\Phi(B) = v$ is equal to \[\E_{S \sim \cs} \mathbbm{1}_{\phi(S)^d = v-u} \P_{A \subset [n]^d \setminus S^d} (\Phi(A) = u).\] For each $1 \le k \le n$, the number of power difference pairs $(A,B)$ with the partition $B = A \cup S^d$ and with $|S|=k$ is equal to $\binom{n}{k} 2^{n^d-k^d}$; for $n$ large, all but a proportion of them that is exponentially small in $n$ satisfy $k \le C(d) \log n$ for some constant $C(d)$ depending only on $d$, so as we will discuss formally later (in Proposition \ref{small support or approximately uniform}), if $\phi$ depends on (say) at least $2C(d) \log n$ coordinates then $P_{A,B}$ becomes approximately proportional to the number of elements of $\F_p$ that that have a $d$-th power equal to $v-u$. For instance, if $d=2$, then that is $1$ if $v-u=0$, $0$ if $v-u \neq 0$ is a quadratic non-residue, and $2$ if $v-u \neq 0$ is a quadratic residue modulo $p$.

As density theorems and their proofs have a long history (starting, perhaps, with Roth’s theorem and its proof \cite{Roth}), of involving a counting lemma which states that assuming that some relevant class of anomalies is avoided, one may find approximately the desired number of instances of the structure that one is looking for, it is of interest to reduce Conjecture \ref{square difference conjecture} to the case where no linear form $\phi: \F_p^n \to \F_p$ leads to a distribution of $\Phi(A)$ with $A \in \ca$ that is substantially different from the distribution of $\Phi(A)$ with $A \in \cp([n]^2$). This is analogous to how in Roth’s proof, a substantial portion of the proof is devoted to reducing to the case where a subset of the integers has no large Fourier coefficients besides the zeroth coefficient.

We now state our third main result. It will show that in attempting to prove (or disprove) Conjecture \ref{polynomial difference conjecture} we may assume that $\ca$ is not substantially distinguishable by linear forms $\F_p^n \to \F_p$ from the entire collection of subsets of $[n]^{d_1} \cup \dots \cup [n]^{d_s}$.

\begin{theorem} \label{no remarkable mod p forms reduction}

For every $\eta>0$, Conjecture \ref{polynomial difference conjecture} is equivalent to the statement $(S_{\eta})$ below. \begin{enumerate}\item [$(S_{\eta})$] Let $d$ be a positive integer and let $\d \in (0,1]$. Then for $n$ large enough depending on $d, \d$ and $\eta$ only, every subset $\ca$ of $\cp([n]^{d})$ that has density at least $\d$ and satisfies \[|\P_{A \in \ca} (\Phi(A)=y) - \P_{A \in \cp([n]^{d})} (\Phi(A)=y)| \le \eta \] for every linear form $\phi:\F_p^n \to \F_p$ and every $y \in \F_p$ contains a pair $(A,B)$ of distinct subsets of $[n]^{d}$ such that $A \subset B$ and $B \setminus A = S^{d}$ for some $S \subset [n]$. \end{enumerate}

\end{theorem}

In Section \ref{Section: Overview of proof methods} we shall sketch the basic principles behind the proofs of our main theorems, and in Section \ref{Section: Proofs of the main results} we shall write these proofs in full.

\section*{Acknowledgement}

The author is grateful to Timothy Gowers for several helpful discussions and suggestions in 2018 and 2019, and in particular for introducing him to the proof technique showing that a dense collection of subsets of $[n]$ eventually contain an interval difference for $n$ large enough.

\section{Overview of the proof methods}\label{Section: Overview of proof methods}

Suppose that we want to show that for some sequence $\Omega_n$ of “universes” indexed by $n$ and some pattern $P$ on pairs of elements (both in the same $\Omega_n$), we have that whenever $\d>0$ is a positive real number, $n$ is large enough (depending on $\d$), and $\ca$ is a subset of $\Omega_n$ with density at least $\d$ in $\Omega_n$ we can always find a pair $(A,B)$ of distinct elements of $\ca$ satisfying $P$.

Then one line of argument runs as follows. Suppose that we can, for large enough $n$, define a non-empty collection $\cw_n$ of subsets of $\Omega_n$ satisfying the four following properties.

\begin{enumerate}[(i)]
\item If $(A,B)$ is a pair of distinct elements of $\Omega_n$ which belongs to the same $\cc \in \cw_n$, then $(A,B)$ satisfies $P$.
\item Every $\cc \in \cw_n$ has the same size $K>0$.
\item Every element of $\Omega_n$ belongs to the same number $L>0$ of collections $\cw_n$.
\item The size of each $\cc \in \cw_n$ tends to infinity with $n$.
\end{enumerate} After that a double-counting argument allows us to conclude as desired: the average density \[ \E_{\cc \in \cw_n} \frac{|\ca \cap \cc|}{|\cc|}\] can by (ii) be rewritten as \[ K^{-1} |\Omega_n| \E_{\cc \in \cw_n} \E_{A \in \Omega_n} \mathbbm{1}_{A \in \ca \cap \cc}\] and then, by changing the order in which we take expectations, as \[ \d K^{-1} |\Omega_n| \E_{A \in \ca} \E_{\cc \in \cw_n} \mathbbm{1}_{A \in \cc}. \] Assumption (iii) then shows that for every $A \in \ca$ the inner expectation is equal to $L/|\cw_n|$, so we obtain \begin{equation} \E_{\cc \in \cw_n} \frac{|\ca \cap \cc|}{|\cc|} = \d (|\Omega_n| L/|\cw_n|K) = \d: \label{Average density equal to d} \end{equation} indeed both products $|\Omega_n|L$ and $|\cw_n|K$ count the number of pairs $(A,\cc) \in \Omega_n \times \cw_n$ satisfying $A \in \cc$, so the parenthetical term is equal to $1$. The identity \eqref{Average density equal to d} states that the average density of $\ca \cap \cc$ inside a random collection $\cc \in \cw_n$ is equal to $\d$, so we can in particular find some $\cc \in \cw_n$ such that the density of $\ca \cap \cc$ inside $\cc$ is at least $\d$. Provided that $n$ is large enough, we can by (iv) find a pair $(A,B)$ of two distinct elements of $\ca$ that both belong to $\cc$, and that pair hence satisfies $P$ by (i).

In many situations, we do not have (i)-(iv) in full, but weaker versions are satisfied which nonetheless suffice, especially for (ii) and (iii). For instance, we may have that most $\cc \in \cw_n$ have approximately the same size, and that most elements of $\Omega_n$ belong to approximately the same number of collections $\cc \in \cw_n$, these two properties often being established using concentration arguments. There, rather than \eqref{Average density equal to d} we instead obtain the sufficient lower bound \begin{equation} \E_{\cc \in \cw_n} \frac{|\ca \cap \cc|}{|\cc|} \ge \tau \d \label{Lower bound on the average density} \end{equation} for some $\tau > 0$ (such as $1/2$).

An example of a setting where the present proof technique works straightforwardly, with (i)-(iv) completely fulfilled is the task of showing that if $\ca$ is a subset of $\Z_2^{\Z_n}$ with density equal to some $\d>0$, and $n$ is large enough depending on $\d$, then we can find distinct $A,B \in \ca$ such that the symmetric difference $A \Delta B$ is an interval modulo $n$. For every $C \in \Z_2^{\Z_n}$ and every $y \in \Z_n$ we define the collection \[\cc(C,y) = \{C, C + \mathbbm{1}_{\{y\}}, C + \mathbbm{1}_{\{y,y+1\}}, C + \mathbbm{1}_{\{y,y+1,y+2\}}, \dots, C + 1\}.\] In this case, the common size of the collections $\cc(C,y)$ is equal to $n$, and \eqref{Average density equal to d} then shows that it suffices that $n > \delta^{-1}$ for $\ca$ to contain a desired pair $(A,B)$.

That (i) is satisfied in this setting follows from the fact that the symmetric difference of two intervals modulo $n$ is an interval modulo $n$, but many patterns $P$ are not “transitive” in the sense that is not the case that if $(A,B)$ and $(A,C)$ satisfy $P$ then $(B,C)$ satisfies $P$: in particular, whenever $d \ge 2$ it is never the case that if $A,B,C$ are subsets of $\Z_2^{[n]^d}$ such that $A \Delta B = X^d$ and $A \Delta C = Y^d$ for some non-empty distinct $X,Y \subset [n]$, then $B \Delta C = Z^d$ for some non-empty $Z \subset [n]$. Likewise if $A,B,C \in \cp([n]^2)$ satisfy $A \subset B \subset C$ with $B \setminus A = Y^d$ and $C \setminus B = X^d$ for some non-empty $X,Y \subset [n]$ then the set difference $C \setminus A$ is never of the type $Z^d$ with $Z \subset [n]$.

The requirement for transitivity is a key limitation of the method that we have described so far, but it nonetheless still provides a recipe for obtaining a distance $2$ version of the pattern. For every $A \in \Omega_n$, we define the collection $\cc(A)$ to be the collection of all $B$ such that $(A,B)$ satisfies $P$; if we can show, using the argument described so far, that $\cc(A)$ contains at least two elements $B,C$ of $\ca$, then we have that $B,C$ are distance $2$ apart for $P$, in the sense that there exists $A$ such that $(A,B)$ and $(A,C)$ both satisfy $P$. This is the basic idea behind how we will prove Theorem \ref{distance 2 theorem} and then Theorem \ref{corollary of distance 2 theorem}. (It will suffice for us to define $\cc(A)$ for some well-chosen “independent” $A \in \Omega_n$ rather than for all $A \in \Omega_n$, and we shall do so to avoid unnecessary technical complications in the proofs of these results.)

Until this point all that we have used from \eqref{Lower bound on the average density} is that (provided that $n$ is large enough) one of the $\cc \in \cw_n$ contains at least two elements of $\ca$. But \eqref{Lower bound on the average density} shows the much stronger claim that $\ca$ has density at least $\tau \d$ inside one of the $\cc \in \cw_n$. This shows that in the distance $2$ results we can not only find pairs $(A,B), (A,C) \in \Omega_n \times \ca$ satisfying $P$, but, for any integer $k$ and $n$ large enough depending on $k$,$\d$ only, pairs $(A,B_1), \dots, (A,B_k) \in \Omega_n \times \ca$.

However, having $\ca \cap \cc$ dense inside some $\cc \in \cw_n$ also allows us to obtain a very different type of statement besides distance $2$ results. Suppose that we aim for the original distance $1$ problem stated in the first paragraph of this section, and that the argument as described so far does not suffice (as we have previously discussed, this is the case whenever $P$ does not exhibit transitivity). Then a variation of the argument may allow us to impose extra structure on the elements of $\ca$.

We say that two sets $V,V'$ are \emph{isomorphic} (for a pattern $P$, which will always be clear given the context) if there exists a bijection $h: V \to V'$ such that for every pair $(A,B)$ of elements of $V$, the pair $(A,B) \in V \times V$ satisfies $P$ if and only if the pair $(h(A),h(B)) \in V' \times V'$ satisfies $P$.

Suppose that for some sequence of $\Omega_n’ \subset \Omega_n$ we know the distance $1$ result analogous to the one that we seek to prove: that is, we know that for every $\d’>0$, for $n$ large enough depending on $\d’$ only every subset $\ca’$ of $\Omega_n’$ with density at least $\d’$ contains a pair $(A,B)$ of elements satisfying $P$. If we can define a sequence $\cw_n$ of collections of subsets of $\Omega_n$ such that, for some $m$ tending to infinity with $n$, each $\cc \in \cw_n$ is isomorphic to some $\Omega_m’$, and the collections $\cw_n$ satisfy (ii), (iii), (iv), then by using \eqref{Lower bound on the average density} and taking $\d’ = \tau \d$ in the assumption above, we conclude that $\ca$ contains a pair $(A,B)$ of elements satisfying $P$. In other words, we have reduced the problem from dense subsets of $\Omega_n$ to dense subsets of $\Omega_m’$. The basic strategy behind much of the proof of Theorem \ref{statements equivalent to the polynomial difference conjecture} will be modelled on this argument, even if only a relaxed version of (iii) will hold.

Even in the second proof approach that we have just described, where we use that $\ca \cap \cc$ is dense in some $\cc \in \cw_n$ (as opposed to containing at least two elements), the lower bound on the density of $\ca \cap \cc$ inside $\cc$ that we obtain is at most the original density of $\ca$. There is however yet a third approach where we can obtain a density increment, that is, we obtain that the difference of the densities \[|\ca \cap \cc|/|\cc| - |\ca|/|\Omega_n|\] is bounded \emph{below} by a positive quantity. That third approach will provide us with a reduction of a different kind to the previous one.

Suppose that for some finite set $Y$ and some class $\cf_n$ of functions from $\Omega_n$ to $Y$, at least for $n$ large enough, we can for every $f \in \cf_n$ constitute a collection $\cw_{n,f}$ satisfying (ii), (iii), (iv) (uniformly over all $f \in \cf_n$) and furthermore satisfying the following three properties.

\begin{enumerate}[(i)]
\setcounter{enumi}{4}
\item For some $m$ tending to infinity with $n$ (uniformly over all $f \in \cf_n$) every $\cc \in \cw_{n,f}$ is isomorphic to some $\Omega_m$.
\item The function $f$ takes a constant value $f(\cc)$ on each $\cc \in \cw_{n,f}$.
\item The distribution of $f(\cc)$ when $\cc \in \cw_{n,f}$ is selected uniformly at random is the same as the distribution of $f(A)$ when $A \in \Omega_n$ is selected uniformly at random.
\end{enumerate} Then for any fixed $\e>0$, we may reduce the problem described in the first paragraph of this section to the same problem with the extra assumption that every $f \in \cf_{n}$ satisfies \[|\P_{A \in \ca}(f(A)=y) - \P_{A \in \Omega_n}(f(A)=y)| \le \e.\] That is, we may assume that every function $f \in \cf_n$ has approximately the same distribution on $\ca$ and on $\Omega_n$. Indeed let $\e>0$, and suppose that there exists some $f \in \cf_n$ and some $y \in Y$ such that \[\P_{A \in \ca}(f(A)=y) - \P_{A \in \Omega_n}(f(A)=y) \ge \e.\] By double-counting, this then ensures that \begin{align*} |\ca \cap \cc|/|\cc| & \ge \d (\P_{A \in \Omega_n}(f(A)=y)+\e)/(\P_{A \in \Omega_n}(f(A)=y)) \\& \ge \d(1+\e)\end{align*} for some fixed $\cc \in \cw_n$. (As in the previous arguments, we first establish this lower bound with an expectation over all $\cc \in \cw_n$ on the left-hand side.) Let $m$ be as in (v) such that the collection $\cc$ is isomorphic to some $\Omega_m$. Letting $n_1=m$, the intersection $\ca \cap \cc$ is then isomorphic to some $\ca^{(1)} \subset \Omega_{n_1}$ with density at least $\d (1+\e)$ in $\Omega_{n_1}$. We then iterate further as long as we can, with a function $f$ that is allowed to change at each iteration, obtaining successive pairs $(\ca^{(j)}, \Omega_{n_j})$ with $|\ca^{(j)}| \ge \d (1+\e)^j |\Omega_{n_j}|$. Because the densities $|\ca^{(j)}|/|\Omega_{n_j}|$ are all at most $1$, the number $t$ of iterations is at most $\log(\d^{-1})/\log(1+\e)$ and we ultimately obtain $(\ca^{(t)}, \Omega_{n_t})$ such that \[|\P_{A \in \ca^{(t)}}(f(A)=y) - \P_{A \in \Omega_{n_t}}(f(A)=y)| \le \e\] for every $f \in \cf_{n_t}$. Because $t$ is bounded above (depending on $\d,\e$ only), and $m$ tends to infinity with $n$ (uniformly), the value $n_t$ is bounded below by some function tending to infinity with $n$ (depending on $\d,\e$ only).

\section{Proofs of the main results}\label{Section: Proofs of the main results}

In this section we provide proofs of our main theorems, Theorem \ref{distance 2 theorem}, Theorem \ref{statements equivalent to the polynomial difference conjecture}, and Theorem \ref{no remarkable mod p forms reduction}.

\subsection{The covering argument}

The main goal of this subsection is to use a covering argument to prove a theorem (Theorem \ref{first two results}) from which we will then deduce both Theorem \ref{distance 2 theorem} and Theorem \ref{statements equivalent to the polynomial difference conjecture}. Informally, we obtain that if $n,m$ are positive integers, $\ca_m$ is a non-empty subset of $\cp([m]^{d_1} \cup \dots \cup [m]^{d_s})$, $\ca$ is a subset of $[n]^{d_1} \cup \dots \cup [n]^{d_s}$ with density at least some $\d>0$, and $n$ is large enough depending on $s,d_1, \dots, d_s, m, \d$ only, then we can find a window of size $m$ such that some subcollection of $\ca$ has all its sets identical outside the window, and the collection of restrictions of this subcollection to the window still has density at least $\d/2$ inside a specified collection $\ca_m$ of possibilities chosen in advance inside that window.

Let $s \ge 1$ and let $d_1, \dots, d_s \ge 1$ be integers. We begin with a statement which asserts that if $t$,$m$ are two integers and $t$ is large enough depending on $m$, then most choices of $A \subset [n]^{d_1} \cup \dots \cup [n]^{d_s}$ will satisfy a property tested on $t$ pairwise disjoint windows of size $m$ for a roughly equal number of such windows. Remarkably, this proposition is uniform with respect to the choice of property, and it is that which makes it versatile enough to enable us to derive from it a wealth of consequences.

Recall the definition of the maps $h_I$ from the introduction in the case where $I \subset [n]$ is an interval. In Proposition \ref{variance bound} that we are about to state and prove we will consider a slight generalisation of these maps to arbitrary subsets $X \subset [n]$ and arbitrary total orderings $<$ on $X$. The reader interested only in the proofs of Theorem \ref{distance 2 theorem}, Theorem \ref{corollary of distance 2 theorem} and Theorem \ref{statements equivalent to the polynomial difference conjecture} may assume that all sets $X_r$ involved in Proposition \ref{variance bound} are intervals and that all orderings are the usual ordering on the integers, as this version will suffice for the purposes of proving these theorems. Nonetheless, we will resort to the more general version later, in order to prove Theorem \ref{no remarkable mod p forms reduction}. If $X \subset [n]$ is a set of size $m$ and $<$ is a total ordering on $X$, then writing the elements of $X$ as $a_1< \dots < a_m$ for this ordering we define for every $B \subset X^{d_1} \cup \dots \cup X^{d_s}$ a subset $h_{(X,<)}(B)$ of $[m]^{d_1} \cup \dots \cup [m]^{d_s}$ by replacing the coordinates $a_1, \dots, a_m$ by $1, \dots, m$ respectively in the elements of $B$.

Throughout Proposition \ref{variance bound}, Theorem \ref{first two results} and their proofs, all expectations, variances and sums over $A$,$A’$ will be taken over the range $[n]^{d_1} \cup \dots \cup [n]^{d_s}$.

\begin{proposition} \label{variance bound} Let $s, d_1, \dots, d_s,m,t$ be positive integers. Let $P: \cp([m]^{d_1} \cup \dots \cup [m]^{d_s}) \to \{0,1\}$ be a property of subsets of $[m]^{d_1} \cup \dots \cup [m]^{d_s}$ which is satisfied by some positive proportion $p(P)$ of subsets of $[m]^{d_1} \cup \dots \cup [m]^{d_s}$. For every $r \in [t]$ let $X_r$ be a subset of $[n]$ with size $m$ and let $<_r$ be a total ordering on $X_r$. Suppose that the sets $X_1, \dots, X_r$ are pairwise disjoint. For every $A \subset [n]^{d_1} \cup \dots \cup [n]^{d_s}$ let $N(A)$ be the number of indices $r \in [t]$ such that \begin{equation} h_{(X_r,<)}(A \cap (X_r^{d_1} \cup \dots \cup X_r^{d_s}))\label{event that is required to satisfy P} \end{equation} satisfies $P$. If $\e>0$ and $t \ge \e^{-1} p(P)^{-1}$ then \[\Var_A N(A) \le \e (\E_A N(A))^2\] when $A$ is chosen uniformly at random in $\cp([n]^{d_1} \cup \dots \cup [n]^{d_s})$. \end{proposition}

\begin{proof} As the sets $X_1, \dots, X_t$ are pairwise disjoint, the sets \[X_1^{d_1} \cup \dots \cup X_1^{d_s}, \dots, X_t^{d_1} \cup \dots \cup X_t^{d_s}\] are also pairwise disjoint, and the events (over $r \in [t]$) that \eqref{event that is required to satisfy P} satisfies $P$ are hence jointly independent. Each of these individual events has probability $p(P)$, so the number $N(A)$ of these events has variance $tp(P)(1-p(P)$); meanwhile the expectation $\E_A N(A)$ is equal to $tp(P)$, so \[ \Var N(A) / (\E_A N(A))^2 = (1-p(P)) / t p(P) \le 1/tp(P).\] The result follows. \end{proof}

We note that if $p(P)>0$, as is the case in the assumption of Proposition \ref{variance bound}, then \[p(P) \ge 2^{-(m^{d_1} + \dots + m^{d_s})}\] and the lower bound $\e^{-1} p(P)^{-1}$ in the condition on $t$ is hence always at most \[2^{(m^{d_1} + \dots + m^{d_s})} \e^{-1};\] this lower bound is therefore (for fixed $s,d_1, \dots, d_s$) bounded above in terms of $m$ and $\e$. In turn, finding $X_1, \dots, X_t$ as required by Proposition \ref{variance bound}, and even requiring them to be intervals, is possible provided that $n \ge mt$, and is therefore in particular possible provided that $n \ge 2^{(m^{d_1} + \dots + m^{d_s})} \e^{-1} m$.

Proposition \ref{variance bound} now allows us to deduce the main result of this subsection, using a double-counting argument.

\begin{theorem}\label{first two results}

Let $s,d_1, \dots, d_s,m$ be positive integers and let $\d>\e>0$. If $n$ is large enough depending on $s,d_1, \dots, d_s, m, \d, \e$ only then whenever $\ca_m$ is some non-empty subset of $[m]^{d_1} \cup \dots \cup [m]^{d_s}$ and $\ca \subset \cp([n]^{d_1} \cup \dots \cup [n]^{d_s})$ has density at least $\d$, there exists an interval $I \subset [n]$ with size $m$ and a subset \[U \subset ([n]^{d_1} \setminus I^{d_1}) \cup \dots \cup ([n]^{d_s} \setminus I^{d_s})\] such that the collection $\cc(I,U)$ defined as \[\{A \in \cp([n]^{d_1} \cup \dots \cup [n]^{d_s}): h_{I}(A \cap (I^{d_1} \cup \dots \cup I^{d_s})) \in \ca_m, A \setminus (I^{d_1} \cup \dots \cup I^{d_s}) = U\}\] satisfies \[|\ca \cap \cc(I,U)| \ge (\d - \e) |\cc(I,U)|.\] In particular, if $n$ is large enough depending on $s,d_1, \dots, d_s, m, \d$ only and $|\ca_m| > 4 \d^{-1}$, then $\cc(I,U)$ contains at least two elements of $\ca$.

\end{theorem}

\begin{proof}

Let $m \ge 1$, and let $n \ge 2^{(m^{d_1} + \dots + m^{d_s})} \e^{-3} m$. We apply Proposition \ref{variance bound} with $P = \mathbbm{1}_{\ca_m}$, with $t = \lfloor n/m \rfloor$ and with $X_1, \dots, X_t$ taken to be the intervals \[I_1 = [m], I_2 = [2m] \setminus [m], \dots, I_t = [tm] \setminus [(t-1)m],\] that is, $N(A)$ is the number of $r \in [t]$ such that \begin{equation} h_{I_r}(A \cap (I_r^{d_1} \cup \dots \cup I_r^{d_s})) \in \ca_m. \label{intersection in defining N(A)} \end{equation} This provides \[\Var N(A) \le \e^3 (\E_A N(A))^2,\] from which we deduce the upper bound \begin{align*} [\P_{A}(N(A) \le (1-\e) \E_{A'} N(A')) & \le \P_{A}(|N(A) -\E_{A'} N(A')| \ge \e \E_{A'} N(A')) \\ & \le (\Var N(A)) / \e^2 (\E_A N(A))^2 \\ & \le \e.\end{align*} Because $\ca$ has density at least $\d>\e$ (in $\cp([n]^{d_1} \cup \dots \cup [n]^{d_s}$), its intersection with the set \[\{A \subset [n]^{d_1} \cup \dots \cup [n]^{d_s}: N(A) \ge (1-\e) \E_{A'} N(A')\}\] has density at least $\d-\e$ (still in $\cp([n]^{d_1} \cup \dots \cup [n]^{d_s})$), which provides the lower bound \begin{equation} \E_A \mathbbm{1}_{A \in \ca} N(A) \ge \E_A \mathbbm{1}_{A \in \ca} \mathbbm{1}_{N(A) \ge (1-\e) \E_{A'} N(A')} N(A) \ge (\d-\e) (1-\e) \E_A N(A). \label{lower bound on pairs in distance 2 proof} \end{equation} For every $r \in [t]$ let $\cd_{n}(I_r)$ be the collection of $A \in \cp([n]^{d_1} \cup \dots \cup [n]^{d_s})$ satisfying \eqref{intersection in defining N(A)}. We have the double-counting equalities \begin{align*} \sum_A N(A) &= \sum_{r \in [t]} |\cd(X)|\\
\sum_A \mathbbm{1}_{A \in \ca} N(A) &= \sum_{r \in [t]} |\ca \cap \cd(X)|. \end{align*} Indeed both sides of the first equality count the number of pairs \[(A,r) \in \cp([n]^{d_1} \cup \dots \cup [n]^{d_s}) \times [t]\] satisfying $A \in \cd(I_r)$, and both sides of the second equality count the number of pairs $(A,r) \in \ca \times [t]$ satisfying $A \in \cd(I_r)$. Using these equalities on the corresponding sides of \eqref{lower bound on pairs in distance 2 proof} then provides \begin{equation} \E_{r \in [t]} |\ca \cap \cd(I_r)| \ge (\d-\e) (1-\e) \E_{r \in [t]} |\cd(I_r)|. \label{first inequality after the double counting in the distance 2 proof} \end{equation} For every $r \in [t]$ we partition the collection $\cd(I_r)$ into collections $\cc(I_r,U)$ defined by \[\cc(I_r,U) = \{A \in \cd(I_r): A \setminus (I_r^{d_1} \cup \dots \cup I_r^{d_s}) = U\}\] for each \begin{equation} U \subset ([n]^{d_1} \setminus I_r^{d_1}) \cup \dots \cup ([n]^{d_s} \setminus I_r^{d_s}). \label{range of U} \end{equation} The inequality \eqref{first inequality after the double counting in the distance 2 proof} then becomes \[\E_{r \in [t], U} |\ca \cap \cc(I_r,U)| \ge (\d-\e) (1-\e) \E_{r \in [t], U}|\cc(I_r,U)|\] where the expectation over $U$ is over all $U$ as in \eqref{range of U}. By the pigeonhole principle we can in particular find some pair $(r,U)$ such that \[|\ca \cap \cc(I_r, U)| \ge (\d-\e) (1-\e) |\cc(I_r,U)|. \]  If we furthermore take $\e = \d/2$ and $m$ such that $|\ca_m| > 4\d^{-1}$ then (since $\cc(I_r,U)$ and $\ca_m$ have the same size) the right-hand side is strictly greater then $1$, so the collection $\cc(I_r,U)$ contains two distinct sets in $\ca$. The result follows. \end{proof}

\subsection{Deductions of the distance 2 result and relative reduction}

In this subsection we use Theorem \ref{first two results} to obtain Theorem \ref{distance 2 theorem}, Theorem \ref{corollary of distance 2 theorem} and Theorem \ref{statements equivalent to the polynomial difference conjecture}. We begin with Theorem \ref{distance 2 theorem}.

\begin{proof} [Proof of Theorem \ref{distance 2 theorem}]

To obtain the first part of Theorem \ref{distance 2 theorem} we take $\cc(I,U)$ as in the conclusion of Theorem \ref{first two results} and take $A$,$B$ to be two elements of $\cc(I,U)$. The stronger conclusion in the case where $\cf$ is nested follows immediately as well. There remains to prove the statement where $I_1, I_2$ are disjoint. To do so, we enumerate $\cf = \{F_1, \dots, F_l\}$ with $l = |\cf|$, and we apply the first part of Theorem \ref{distance 2 theorem} to $ml$ instead of $m$ and to the collection $\cf’ \subset \cp([ml]^{d_1} \cup \dots \cup [ml]^{d_s})$ obtained by inserting the sets $F_1, \dots, F_l$ in “diagonal blocks”. More formally, we consider the collection \[\cf’ = \{h_{[tm] \setminus [(t-1)m]}(F_t): t \in [l]\}\] of pairwise disjoint sets. \end{proof}

Having proved Theorem \ref{distance 2 theorem}, we are able to quickly derive Theorem \ref{corollary of distance 2 theorem} in turn.

\begin{proof}[Proof of Theorem \ref{corollary of distance 2 theorem}]

To obtain the first part of Theorem \ref{corollary of distance 2 theorem} we take \[\cf = \{S^{d_1} \cup \dots \cup S^{d_s}: S \subset [m]\}\] in the first part of Theorem \ref{distance 2 theorem}. To require $S_1$, $S_2$ in Theorem \ref{corollary of distance 2 theorem} to both have size $m$ we instead take \[\cf = \{S^{d_1} \cup \dots \cup S^{d_s}: S \subset [2m], |S|=m\}\] in the first part of Theorem \ref{distance 2 theorem}, and to furthermore require $S_1$, $S_2$ to be disjoint we take this same $\cf$ and apply the second part of Theorem \ref{distance 2 theorem}. Finally, to obtain $S_1 \subset S_2$ in Theorem \ref{corollary of distance 2 theorem} we take \[\cf = \{[t]^{d_1} \cup \dots \cup [t]^{d_s}: t \in [m]\}\] in the last part of Theorem \ref{distance 2 theorem}. \end{proof}

In the remainder of this subsection we deduce Theorem \ref{statements equivalent to the polynomial difference conjecture} from Theorem \ref{first two results}. Theorem \ref{statements equivalent to the polynomial difference conjecture} asserts a sequence of equivalences between six statements. Although there will be some deviations from that it will be convenient for us to largely show the equivalence first between Conjecture \ref{polynomial difference conjecture} and (i), then between (i) and (ii), and so forth until (v).

\begin{proof}[Proof of Theorem \ref{statements equivalent to the polynomial difference conjecture}]

Item (i) follows from Conjecture \ref{polynomial difference conjecture} as a special case. To prove the converse, it suffices to show that if Conjecture \ref{polynomial difference conjecture} holds for some $(d_1, \dots, d_s)$, then it holds for any $(d_1’, \dots, d_s’)$ such that $d_j’ \le d_j$ for every $j \in [s]$. To see this, we proceed as follows. We define for every pair $(d’,d’)$ with $0 \le d’ \le d$ a map $i_{d’,d}: [n]^{d’} \to [n]^d$ by \[i_{d’,d}(x_1, \dots, x_{d’}) = (x_1, \dots, x_1, x_2, x_{d’})\] where the first coordinate is repeated $d-d’+1$ times. If $A’ = A_1’ \cup \dots \cup A_s’$ is a subset of $[n]^{d_1'} \cup \dots \cup [n]^{d_s’}$ then we define \[i(A) = i_{d_1’,d_1}(A_1’) \cup \dots \cup i_{d_s’,d_s}(A_s’).\] We let $E = i([n]^{d_1’} \cup \dots \cup [n]^{d_s’})$. If $\ca’$ is a collection of subsets of $[n]^{d_1’} \cup \dots \cup [n]^{d_s’}$ then we define $\ca$ to be the collection of subsets $A = A_1 \cup \dots \cup A_s \subset [n]^{d_1} \cup \dots \cup [n]^{d_s}$ such that $A \cap E = i(A’)$ for some $A’ \in \ca'$. The density of the resulting collection $\ca$ inside $[n]^{d_1} \cup \dots \cup [n]^{d_s}$ is the same as the density of $\ca’$ inside $[n]^{d_1’} \cup \dots \cup [n]^{d_s’}$; if that is at least some $\d>0$ and $n$ is large enough (depending on $s,d_1, \dots, d_s, \d$ only) then Conjecture \ref{polynomial difference conjecture} for $(d_1, \dots, d_s)$ provides some pair $(A,B)$ of sets in $\ca$ such that $A \subset B$ and $B \setminus A = S^{d_1} \cup \dots \cup S^{d_s}$ for some non-empty $S \subset [n]$. The pair $(A’,B’)$ defined by $A’ = i^{-1}(A \cap E)$ and $B’ = i^{-1}(B \cap E)$ then satisfies $A’ \subset B’$ and $B’ \setminus A’ = S^{d_1'} \cup \dots \cup S^{d_s’}$.

Next, (ii) follows from (i) as a special case. To prove the converse implication, we apply Theorem \ref{first two results} with $d_1, \dots, d_s = d$ and with \[\ca_m = \{A_1 \cup \dots \cup A_s \in [m]^{d} \cup \dots \cup [m]^{d}: A_1 = \dots = A_s\}\] for every $m \ge 1$. Let $\d>0$ be fixed, then let $m$ be large enough but fixed, such that (ii) holds for this $\d$ (with $n$ instead of $m$). If $\ca \subset [n]^{d_1} \cup \dots \cup [n]^{d_s}$ has density at least $2\d$, then Theorem \ref{first two results} shows that for $n$ large enough (depending on $s,d,\d,m$ only) we have \[|\ca \cap \cc(I,U)| \ge \d |\cc(I,U)|\] for some $I$,$U$, with $\cc(I,U)$ as in Theorem \ref{first two results}. The collection $\ca \cap \cc(I,U)$ is isomorphic to a collection \[\ct = \{A_0 \cup \dots \cup A_0: A_0 \in \ca_0\}\] for some $\ca_0 \subset \cp([m]^d)$ with density at least $\d$ inside $\cp([m]^d)$, and applying (ii) then provides a pair $(A_0, B_0)$ of sets in $\ca_0$ such that $A_0 \subset B_0$ and $B_0 \setminus A_0 = S^d$ for some non-empty $S \subset [m]$. The sets $A_0 \cup \dots \cup A_0$ and $B_0 \cup \dots \cup B_0$ hence belong to $\ct$, so we obtain a pair $(A,B)$ of sets in $\ca$ that satisfies the conclusion of (i).

After that we show the equivalence between (ii) and (iii). Choosing $\ca_m = [m]^{d}$ for every $m \ge 1$ in (ii) establishes (iii), so it suffices to show that (iii) implies (ii). Suppose that (iii) holds for some sequence $(\ca_m)_{m \ge 1}$. Let $\d>0$ and let $m \ge 1$ be such that every subset of $\ca_m$ with density at least $\d/2$ inside $\ca_m$ contains a power difference pair. Let $\ca$ be a subset of $[n]^{d}$ with density at least $\d$. Applying Theorem \ref{first two results} with $\e = \d/2$ to $\ca_m$ and $\ca$ we obtain that for $n$ large enough (independently of $\ca$) there exists some $\cc(I,U)$ such that \[|\ca \cap \cc(I,U)| \ge (\d/2) |\cc(I,U)|.\] The collection $\cc(I,U)$ is isomorphic to $\ca_m$, so by our assumption on $m$ the collection $\ca \cap \cc(I,U)$ contains a power difference pair.

We then prove the equivalence between (iv) and (ii). We note that (iv) immediately implies (iii) and therefore (ii), and that although this is the simpler of the two implications, this is also certainly the more useful from the viewpoint of aiming to prove Conjecture \ref{polynomial difference conjecture}. To obtain the converse direction we will use the set \[R = \{(x_1, \dots, x_d) \in [n]^d: x_1 \le \dots \le x_d\},\] as every symmetric set is determined by its restriction to $R$. If $\ca$ is a subset of $\cp([n]^d)_{\mathrm{Sym}}$ then the set \[\ca’ = \{A’ \in \cp([n]^d): A’ \cap R \in \ca\}\] has the same density inside $\cp([n]^d)$ as $\ca$ has inside $\cp([n]^d)_{\mathrm{Sym}}$; if that is at least some fixed $\d$ then for $n$ large enough depending on $d,\d$ only (ii) provides a power difference pair $(A_0,B_0)$ of sets in $\ca’$, and the symmetric sets $A$,$B$ which respectively coincide with $A_0$,$B_0$ on $R$ hence also constitute a power difference pair.

Finally we show that Conjecture \ref{polynomial difference conjecture} and (v) are equivalent. Let $s,d_1,\dots,d_s$ be fixed. We define \[T_j = \{(x_1, \dots, x_{d_j}) \in [n]^{d_j}: x_1 < \dots < x_{d_j}\} \] for every $j \in [s]$. If $\cg$ is a collection of elements $G = G_1 \cup \dots \cup G_s$ of $\cg(K([n],d_1) \cup \dots \cup K([n],d_s))$ then we define a collection $\ca$ of subsets $A = A_1 \cup \dots \cup A_s \subset [n]^{d_1} \cup \dots \cup [n]^{d_s}$ by $A \in \ca$ if and only if for some $G \in \cg$ we have \[G_j = \{\{x_1,\dots,x_{d_j}\}: (x_1,\dots,x_{d_j}) \in A_j \cap T_j\}\] for every $j \in [s]$. The collection $\ca$ has the same density as $\cg$ has, so if this density is some fixed $\d$ and $n$ is large enough depending on $\d$ only, then Conjecture \ref{polynomial difference conjecture} provides a polynomial difference pair of elements of $\ca$ which then establishes (v) by the definition of $\ca$ in terms of $\cg$.

Conversely suppose (v) and let us prove Conjecture \ref{polynomial difference conjecture}. To do so we will prove (iv), which suffices. Let $d$ be a positive integer. We begin by writing the set $R$ defined above as the disjoint union \[\bigcup_{1 \le k \le d} \bigcup_{P \in \ci_k} R_P\] where $\ci_k$ is the set of partitions of $[d]$ into $k$ (non-empty) intervals, and for every $P \in \ci_k$ the set $R_{P}$ is the set of $x \in R$ such that $x_i$, $x_j$ are the same if and only if $i$,$j$ belong to the same part of the partition $P$. For each $k \in [d]$ we enumerate $\ci_k = \{P_1, \dots, P_{m(k)}\}$, where $m(k) = |\ci_k|$. We now construct a bijection $\beta$ from $\cp(R)$ to $\cg(K([n],d_1) \cup \dots \cup K([n],d_s))$ where $s = |\ci_1| + \dots + |\ci_d|$ and exactly $|\ci_k|$ indices among $d_1, \dots, d_s$ are equal to $k$ for every $k \in [d]$. If $A$ is a symmetric subset of $[n]^{d}$ then we define \[\beta(A) = G_1 \cup \dots \cup G_s \in \cg(K([n],d_1) \cup \dots \cup K([n],d_s))\] as follows. For every $k \in [d]$ and every $t \in [m(k)]$ we write $P_t = \{P_{t,1}, \dots, P_{t,k}\}$. Then, for every $k$-tuple $(a_1, \dots, a_k)$ of distinct elements of $[n]$ satisfying $a_1 \le \dots \le a_k$, we include the hyperedge $\{a_1, \dots, a_k\}$ in $G_{m(1)+\dots+m(k-1)+t}$ if and only if the elements with coordinates $a_1, \dots, a_k$, repeated respectively $|P_1|, \dots, |P_k|$ times, belong to $A$, that is, if and only if the one such element of $R$ belongs to $A$. If $\ca$ is a collection of symmetric subsets of $[n]^{d}$ then $\beta(\ca)$ has the same density inside $\cg(K([n],d_1) \cup \dots \cup K([n],d_s))$ as $\ca$ has inside $([n]^d)_{\mathrm{Sym}}$; assuming (v) then provides (for $n$ large enough depending on $d$, $\d$ only) some pair $(G,H)$ satisfying the conclusion of (v), and the pair $(A,B)$ defined by $A = \beta^{-1}(H)$, $B = \beta^{-1}(G)$ is then a power difference pair. \end{proof}

\subsection{Deduction of the quasirandomness reduction}

In this subsection we prove Theorem \ref{no remarkable mod p forms reduction}, which unlike Theorem \ref{distance 2 theorem} and Theorem \ref{statements equivalent to the polynomial difference conjecture} does not follow from Theorem \ref{first two results} itself and requires additional preparation. The proof of Theorem \ref{no remarkable mod p forms reduction} will however rely on Proposition \ref{variance bound}, as the proof of Theorem \ref{first two results} does.

It will be convenient for us to use the following statement, which follows immediately from \cite{Gowers and K approximation}, Proposition 2.2 and states that unless a linear form $\F_p^n \to \F_p$ depends on few elements, it must have approximately uniform distribution on $\{0,1\}^n$. If $\phi: \F_p^n \to \F_p$ is a linear form then we say that the \emph{support} $Z(\phi)$ of $\phi$ is the set \[ \{z \in [n]: a_z \neq 0\}.\]

\begin{proposition}\label{small support or approximately uniform}

Let $p$ be a prime, and let $\phi: \F_p^n \to \F_p$ be a linear form. Then \[|\P_{A \in \{0,1\}^n}(\phi(A)=y) - 1/p| \le p(1-p^{-2})^{|Z(\phi)|}\] for every $y \in \F_p$.

\end{proposition}

Recall the definition \eqref{definition of Phi in terms of phi} of the linear form $\Phi: \F_p^{[n]^{d}} \to \F_p$ in terms of the linear form $\phi: \F_p^{n} \to \F_p$ from the introduction. In the following proposition and the remainder of this subsection, if $X_{i,1}, \dots, X_{i,m}$ are pairwise disjoint subsets of $[n]$ for some integer $i$ then we shall write $X_i$ for their union $X_{i,1} \cup \dots \cup X_{i,m}$. If furthermore $U \subset [n]^d \setminus X_i^d$, then we define $\cc(X_i, U)$ as the collection of sets $A \subset [n]^{d}$ such that $1_A$ is constant on each of the sets that are of one of the types \[X_{i,j_1} \times \dots \times X_{i,j_d}\] with $j_1, \dots, j_d \in [m]$ (not necessarily distinct), and which furthermore satisfy $A \setminus X_i^{d} = U$.

\begin{proposition}\label{three properties}

Let $d, m$ be positive integers, let $\e,\eta>0$, let $p$ be a prime and let $\phi: \F_p^n \to \F_p$ be a linear form. If $n$ is large enough depending on $d,m,\e,\eta,p$ only then there exists a partition \[\left( \bigcup_{1 \le i \le t} \bigcup_{1 \le j \le m} X_{i,j} \right) \cup R\] of $[n]$ such that $t \ge (n/pm) - 2$, the sets $X_{i,j}$ all have the same size equal to at most $p$, and the collections $\cc(X_i, U)$ with $i \in [t]$ and $U \subset [n]^d \setminus X_i^d$ satisfy the following three properties.

\begin{enumerate}

\item The linear form $\Phi$ takes a constant value $\Phi(X_i, U)$ on each of the collections $\cc(X_i, U)$.

\item We have the approximation \begin{equation} |\P_{(X_i,U)}(\Phi(X_i, U) = y) - \P_{A}(\Phi(A) = y)| \le \eta \label{about the same distribution} \end{equation} for every $y \in \F_p$, where the first probability is over all $(X_i,U)$ as above and the second is over all sets $A \in \cp([n]^{d})$.

\item The number $N(A)$ of collections $\cc(X_i, U)$ to which a random set $A \in \cp([n]^{d})$ belongs satisfies \[\Var_A N(A) \le \e (\E_A N(A))^2.\]

\end{enumerate}

\end{proposition}

\begin{proof}

We first note that in order for $\Phi$ to be constant on each $\cc(X_i, U)$, it suffices that \[\phi(X_{i,j}) = \sum_{z \in X_i} a_z\] vanishes for every $(i,j) \in [t] \times [m]$. Indeed, any contribution \[\Phi(X_{i,j_1} \times \dots \times X_{i,j_d})\] with $j_1, \dots, j_{d} \in [m]$ (not necessarily distinct) factors as \[\phi(X_{i,j_1}) \dots \phi(X_{i,j_d});\] if the factors all vanish, then in particular the product vanishes. 

We then distinguish two cases. If the support $Z(\phi)$ of $\phi$ contains at most $n/2$ elements, then we collect the elements of $[n] \setminus Z(\phi)$ into pairwise disjoint sets $X_{i,1}, \dots, X_{i,m}$ each of size $1$ for $i=1,2,\dots$, until there are at most $m-1$ remaining elements. We then take the “remainder set” $R$ to be the union of $Z(\phi)$ and of these remaining elements. If on the other hand $Z(\phi)$ contains $n/2+1$ elements or more, then we take some subset $Y$ of $Z(\phi)$ with size between $n/4$ and $3n/8$, and constitute a subset $X_{1,1} \subset Y$ of $p$ elements such that all coefficients of $\phi$ at the elements of $X_{1,1}$ are equal and hence ensure $\phi(X_{1,1}) = 0$, then another disjoint subset $X_{1,2} \subset Y \setminus X_{1,1}$ of $p$ elements such that all coefficients of $\phi$ at the elements of $X_{1,2}$ are equal and hence ensure $\phi(X_{1,2}) = 0$, and so forth until we have defined $X_{1,1}, \dots, X_{1,m}$. We then iterate inside $Y \setminus X_1$, defining sets $X_{2,1}, \dots, X_{2,m}$, then continue inside $Y \setminus (X_1 \cup X_2)$, and so forth: for successive $j=1,2, \dots$ we define the sets $X_{j,1}, \dots, X_{j,m}$ until the size of $Y \setminus (X_1 \cup \dots \cup X_j)$ drops below $p^2+mp$ (as until then the procedure that we have described is guaranteed to work). We then take $R$ to be the union of $[n] \setminus Y$ and of the elements of $Y \setminus (X_1 \cup \dots \cup X_j)$. In both cases we have $\phi(X_{i,j}) = 0$ for each $(i,j) \in [t] \times [m]$, which establishes the first condition.

We next establish \eqref{about the same distribution}. In both cases, because every collection $\cc(X_i, U)$ contains the same number of sets, the distribution of $\Phi(X_i, U)$ with $(X_i, U)$ chosen at random is equal to the distribution $D$ of $\Phi(U)$ where we first choose $i \in [t]$ uniformly at random, and then choose the subset $U$ of $[n]^{d} \setminus X_i^{d}$ uniformly at random. In the first case, all elements of the support \[Z(\Phi) = \{(z_1, \dots, z_d) \in [n]^d: a_{z_1} \dots a_{z_d} \neq 0\}\] of $\Phi$ belong to $[n]^{d} \setminus X_i^{d}$ so the distribution $D$ is exactly the distribution of $\Phi(A)$ when $A$ is chosen uniformly at random in $\cp([n]^d)$. In the second case, the intersection $Z(\Phi) \cap ([n]^{d} \setminus X_i^{d})$ contains $(Z(\phi) \cap R)^{d}$ which has size at least $(n/8)^{d} \ge n/8$. Proposition \ref{small support or approximately uniform} hence shows that assuming that $n$ is large enough, for every choice of $i \in [t]$ the distribution $D$ conditioned on $i$ charges every element of $\F_p$ with probability mass at most $\e/2$ away from $1/p$. By the law of total probability this is the case for the distribution $D$ itself, so for that of $\Phi(X_i, U)$. Meanwhile, since $|Z(\Phi)| \ge |Z(\phi)| \ge n/2$, Proposition \ref{small support or approximately uniform} shows that this is the case as well for the distribution of $\Phi(A)$ as well. The bound \eqref{about the same distribution} then follows from the triangle inequality.

Let $\sigma \in \{1,p\}$ be the common size of the sets $X_{i,j}$. The last item follows from applying Proposition \ref{variance bound}: there, we take the property $P: \cp([\sigma m]^d) \to \{0,1\}$ to be that the subset of $[\sigma m]^d$ has a constant indicator function on each product \[ [\sigma (j_1-1)+1, \sigma j_1] \times \dots \times [\sigma (j_d-1)+1, \sigma j_d] \] with $j_1, \dots, j_d \in [m]$ (not necessarily distinct), and we take the pairs $(X_i, <_i)$ to be such that for every $i \in [t]$ the total ordering $<_i$ has the elements of $X_{i,1}$ as its $\sigma$ lowest elements, then the elements of $X_{i,2}$ as its $\sigma$ lowest elements after that, and so forth.  \end{proof}

The next statement provides us with a density increment for dense collections $\ca \subset \cp([n]^d)$ that are distinguishable from $\cp([n]^d)$ using some linear form $\F_p^n \to \F_p$. In Proposition \ref{density increment} we assume that the integer $m$ is fixed and that the sets $X_1, \dots, X_t$ are as in the conclusion of Proposition \ref{three properties}.

\begin{proposition}\label{density increment}

Let $d$ be a positive integer, let $\d>0$, $\eta \in (0,1]$, and let $p$ be a prime. Suppose that $\ca$ is a collection of subsets of $[n]^{d}$ that has density $\d$, and that there exists a linear form $\phi: \F_p^n \to \F_p$ such that such that the corresponding linear form $\Phi: \F_p^{[n]^d} \to \F_p$ satisfies \[|\P_{A \in \ca}(\Phi(A)=y) - \P_{A \in \cp([n]^{d})} (\Phi(A)=y)| \ge \eta\] for some $y \in \F_p$. If $n$ is large enough depending on $d,m,\d,\eta,p$ only then there exists $i \in [t]$ and a subset $U \subset [n]^{d} \setminus X_i^{d}$ such that \[|\ca \cap \cc(X_i, U)| \ge \d (1+\eta /3) |\cc(X_i, U)|.\]

\end{proposition}

\begin{proof}

We write $P$ for $\P_{A \in \cp([n]^{d})}(\Phi(A)=y)$. The set $\ca_y$ defined by \[\ca_y = \{A \in \ca: \Phi(A) = y\}\] has density at least $\d(P+\eta)$ inside $\cp([n]^{d})$, so by the inequality \eqref{lower bound on pairs in distance 2 proof} from the deduction of Theorem \ref{first two results} from Proposition \ref{variance bound}, applied to $\ca_y$ instead of $\ca$ and to $\e = \d \eta /3$, we have for $n$ large enough depending on $d,m,\d,\eta$ only the lower bound \begin{align*} \E_A \mathbbm{1}_{A \in \ca_y} N(A) & \ge (\d(P+\eta) - \d \eta /6) (1-\d \eta /6) \E_A N(A)\\ & \ge \d (P + 2\eta /3) \E_A N(A). \end{align*} Using the double-counting identities \begin{align*} \sum_A N(A) &= \sum_{(i, U)} |\cc(X_i, U)|\\
\sum_A \mathbbm{1}_{A \in \ca_y} N(A) &= \sum_{(i, U)} \mathbbm{1}_{\Phi(X_i, U) = y} |\ca \cap \cc(X_i, U)| \end{align*} shows that \[\E_{(i, U)} \mathbbm{1}_{\Phi(X_i, U) = y} |\ca \cap \cc(X_i, U)| \ge \d (P + 2\eta /3) \E_{(i, U)} |\cc(X_i, U)|.\] By the second property guaranteed by Proposition \ref{three properties}, the proportion of $(i, U)$ for which $\Phi(X_i, U) = y$ is at most $P + \eta/3$, so for $n$ large enough depending on $d,m,\d,\eta,p$ only we have \begin{align*} \E_{(i, U): \mathbbm{1}_{\Phi(X_i, U) = y}} |\ca \cap \cc(X_i, U)| & \ge \d \frac{P + 2\eta /3}{P + \eta /3} \E_{(i, U)} |\cc(X_i, U)| \\ & \ge \d (1+\eta /3)\E_{(i, U)} |\cc(X_i, U)| \end{align*} since $P+\eta \le 1$. As all collections $\cc(X_i, U)$ have the same size, the expectation in the last quantity may be restricted to the $(i, U)$ satisfying $\Phi(X_i, U) = y$. It follows that \[|\ca \cap \cc(X_i, U)| \ge \d (1+\eta /3) |\cc(X_i, U)|\] for some $(i, U)$. \end{proof}

Iterating this density increment finally allows us to deduce Theorem \ref{no remarkable mod p forms reduction}.

\begin{proof}[Proof of Theorem \ref{no remarkable mod p forms reduction}]

For every $\eta>0$, the statement $(S_{\eta})$ follows from Conjecture \ref{polynomial difference conjecture} as a special case, so it suffices to prove the converse direction. Let $d \ge 1, \d>0$ be fixed, and let $\eta>0$ be a value that verifies $(S_{\eta})$ for these parameters. If we have \[|\P_{A \in \ca} (\Phi(A)=y) - \P_{A \in \cp([n]^{d})} (\Phi(A)=y)| \le \eta \] for every linear form $\phi:\F_p^n \to \F_p$ and every $y \in \F_p$ then we are done by $(S_{\eta})$. Otherwise, we may take $\phi: \F_p^n \to \F_p$ and $y \in \F_p$ such that \[|\P_{A \in \ca} (\Phi(A)=y) - \P_{A \in \cp([n]^{d})} (\Phi(A)=y)| \ge \eta/p, \] and Proposition \ref{density increment} then provides a collection $\cc(X_i, U)$ such that \[|\ca \cap \cc(X_i, U)| \ge \d (1+\eta /3p) |\cc(X_i, U)|.\] The collection $\cc(X_i, U)$ can be seen to be isomorphic to $\cp([n^{(1)}]^{d})$ with $n^{(1)} = m$, using the map $h: \cp([n^{(1)}]^{d}) \to \cc(X_i, U)$ defined by \[h(B) = U \cup \bigcup_{(j_1, \dots, j_d) \in B} (X_{i,j_1} \times \dots \times X_{i,j_d}),\] so we have obtained that $\ca$ contains a subset which is isomorphic to a subset $\ca^{(1)}$ of $\cp([n^{(1)}]^{d})$ with density at least $\d (1+\eta /3p)$ in $\cp([n^{(1)}]^{d})$. All that we have described so far holds for $n$ large enough depending (for fixed $\d,\eta,p$) only on $m$. As $n$ tends to infinity, we may hence also make $m$ tend to infinity, and therefore $n^{(1)}$ as well. We then iterate: if \[|\P_{A \in \ca^{(1)}} (\Phi(A)=y) - \P_{A \in \cp([n^{(1)}]^{d})} (\Phi(A)=y)| \le \eta \] for every linear form $\phi:\F_p^{n^{(1)}} \to \F_p$ and every $y \in \F_p$ then we are done by $(S_{\eta})$, and otherwise we obtain that $\ca^{(1)}$ (and hence $\ca$) contains a subset isomorphic to some subset $\ca^{(2)}$ of $\cp([n^{(2)}]^{d})$ with density at least $\d (1+\eta /3p)^2$, where $n^{(2)}$ tends to infinity with $n$. After any $q$ iterations the density of $\ca^{(q)}$ is at least $\d (1+\eta /3p)^q$; as it must be at most $1$ we have $q \le \log (\d^{-1}) / \log(1+\eta/3p)$. Therefore, we can find $n’$ tending to infinity with $n$ in a way that depends only on $\d$, $\eta$, $p$ such that $\ca$ contains a subset isomorphic to some subset $\ca’$ of $\cp([n’]^{d})$ with density at least $\d$ and which satisfies \[|\P_{A \in \ca’} (\Phi(A)=y) - \P_{A \in \cp([n']^{d})} (\Phi(A)=y)| \le \eta \] for every linear form $\phi:\F_p^{n’} \to \F_p$ and every $y \in \F_p$. We conclude by applying $(S_{\eta})$. \end{proof}


\begin{thebibliography}{9}

\bibitem{Alon}

N. Alon, \textit{Graph-codes}, European J. Combin, \textbf{116}, (2024), Paper No. 103880.

\bibitem{Alweiss}

R. Alweiss, \textit{On a Conjecture of Gowers on Clique Differences}, arxiv:2011.04039 (2020).

\bibitem{Bergelson Leibman}

V. Bergelson and A. Leibman. \textit{Polynomial extensions of van der Waerden’s and Szemer\'edi’s theorems}, J. Amer. Math. Soc. \textbf{9} (1996), 725–753.

\bibitem{Campos Griffiths Morris Sahasrabudhe}

M. Campos, S. Griffiths, R. Morris, J. Sahasrabudhe, \textit{An exponential improvement for diagonal Ramsey}, arXiv:2303.09521 (2024).

\bibitem{Furstenberg and Katznelson k=3}

H. Furstenberg and Y. Katznelson, \textit{A density version of the Hales-Jewett theorem for $k=3$}, Discrete Math. \textbf{75} (1989), 227-241.

\bibitem{Furstenberg and Katznelson}

H. Furstenberg and Y. Katznelson, \textit{A density version of the Hales-Jewett theorem}, J. Anal. Math. \textbf{57} (1991), 64-119.

\bibitem{Gowers}

W. T. Gowers, \textit{The first unknown case of polynomial DHJ}, \url{https://gowers.wordpress.com/2009/11/14/the-first-unknown-case-of-polynomial-dhj/}.

\bibitem{Gowers and K approximation}

W. T. Gowers and T. Karam, \textit{Low-complexity approximations for sets defined by generalizations of affine conditions}, arXiv:2306.00747 (2023).

\bibitem{Green and Tao}

B. Green, T. Tao, \textit{The primes contain arbitrarily long arithmetic progressions}, Ann. Math. \textbf{167} (2008), 481–547.

\bibitem{Gupta Ndiaye Norin Wei}

P. Gupta, N. Ndiaye, S. Norin, L. Wei, \textit{Optimizing the CGMS upper bound on Ramsey numbers}, arXiv:2407.19026 (2024).

\bibitem{Hales-Jewett}

A. W. Hales and R. I. Jewett, \textit{Regularity and positional games}, Trans. Amer. Math. Soc. \textbf{106} (1963), 222-229.

\bibitem{Peluse}

S. Peluse, \textit{Subsets of $\mathbb{F}_p^n \times \mathbb{F}_p^n$ without L-shaped configurations}, Compos. Math. \textbf{160} (2024), 176-236.

\bibitem{Polymath}

D. H. J. Polymath, \textit{A new proof of the density Hales-Jewett theorem}, Ann. Math. \textbf{175} (2012), 1283-1327.

\bibitem{Ramsey}

F. P. Ramsey, \textit{On a Problem of Formal Logic}, Proc. London Math. Soc. \textbf{30} (1930), 264–286.

\bibitem{Roth}

K. F. Roth, \textit{On certain sets of integers}, J. London Math. Soc. \textbf{28} (1953), 245–252.

\bibitem{Sperner}

E. Sperner, \textit{Ein Satz {\"u}ber Untermengen einer endlichen Menge}, Math. Z. \textbf{27} (1928), 544-548.

\bibitem{Szemeredi}

E. Szemerédi, \textit{On sets of integers containing no k elements in arithmetic progression}, Acta Arith. \textbf{27} (1975), 299-345.

\bibitem{van der Waerden}

B. L. van der Waerden, \textit{Beweis einer Baudetschen Vermutung}, Nieuw Archief voor Wiskunde \textbf{15} (1927), 212-216.


\end{thebibliography}
\end{document}